\documentclass{amsart}
\usepackage{mathtools}     
\usepackage{amssymb}
\usepackage[english]{babel}

\newtheorem{theorem}{Theorem}[section]
\newtheorem{lemma}[theorem]{Lemma}

\newtheorem*{conjecture}{Conjecture}
\theoremstyle{remark}
\newtheorem*{remark}{Remark}

\numberwithin{equation}{section}
\newcommand{\dN}{\mathbb N}                             
\DeclareMathOperator{\e}{\mathrm{e}}                    
\DeclareMathOperator{\dens}{\,\mathrm dens}             
\allowdisplaybreaks
\begin{document}
\title[]{Approaching Cusick's conjecture on the sum-of-digits function}
\author{Lukas Spiegelhofer}
\address{Vienna University of Technology,
Vienna, Austria}
\thanks{
The author acknowledges support by the joint project MuDeRa
between the Agence Nationale de la Recherche (ANR, France) and the FWF, project number I-1751-N26.}
\begin{abstract}
Cusick's conjecture on the binary sum of digits $s(n)$ of a nonnegative integer $n$ states the following:
for all nonnegative integers $t$ we have
\[
c_t=\lim_{N\rightarrow\infty}
\frac 1N
\left\lvert
\{n<N:s(n+t)\geq s(n)\}
\right\rvert
>1/2.
\]

We prove that for given $\varepsilon>0$
we have
\[c_t+c_{t'}>1-\varepsilon\]
if the binary expansion of $t$ contains enough blocks of consecutive $\mathtt 1$s (depending on $\varepsilon$),
where $t'=3\cdot 2^\lambda-t$ and $\lambda$ is chosen such that $2^\lambda\leq t<2^{\lambda+1}$.

\end{abstract}
\maketitle
\section{Introduction}
The binary sum-of-digits function $s$ is defined by
\[s\left(\varepsilon_\nu 2^\nu+\cdots+\varepsilon_02^0\right)
=\varepsilon_\nu+\cdots+\varepsilon_0\]
for all digits $\varepsilon_i\in\{0,1\}$.
It is an elementary yet difficult problem to consider the behaviour of $s$ under addition of a constant.
T.~W.~Cusick (private communication) proposed the following conjecture.
\begin{conjecture}
For a nonnegative integer $t$, define
\begin{equation*}
c_t=\dens\{n\geq 0:s(n+t)\geq s(n)\},
\end{equation*}
where $\dens A$ denotes the asymptotic density of a set $A\subseteq \dN$.
Then $c_t>1/2$.
\end{conjecture}
We note that the set in question is in fact a finite union of arithmetic progressions, so there are no problems of convergence (see B\'esineau~\cite{B1972}). 
This conjecture arose when Cusick was working on a related conjecture due to Tu and Deng~\cite{TD2011,TD2012}, which concerns binary addition modulo $2^k-1$.
Tu and Deng's conjecture is of interest since it allows constructing Boolean functions with desirable cryptographic properties.
Partial results on the Tu--Deng conjecture have been obtained, see for example~\cite{CLS2011,DY2012,FRCM2010,QSF2016}.
Moreover, both Cusick's conjecture and the Tu--Deng conjecture have been proven \emph{asymptotically}, the former by Drmota, Kauers, and the author~\cite{DKS2016}, and the latter by Wallner and the author~\cite{SW2019}, but the full statements are still open.
In particular, we wish to note that no bound of the form $c_t>a$ or $c_t<b$ for some $a>0$ or $b<1$, valid for all $t\geq 0$, is known!
In this paper we concentrate on Cusick's conjecture.
The abovementioned result by Drmota, Kauers and the author~\cite{DKS2016}
is the following: we have $c_t>1/2$ for almost all $t$ in the sense of asymptotic density, that is,
\[\dens\{t:c_t>1/2\}=1.\]
However, this theorem does not tell us anything about the structure of such a set of ``good'' $t$; it does not provide a statement allowing to extract many examples of integers $t$ satisfying Cusick's conjecture.

The present paper constitutes a step in this direction.
More precisely, since the original conjecture is elusive and too hard, we consider a simplified version. In order to formulate this easier statement and our main theorem, we define
$t'=3\cdot 2^\lambda-t$, where $2^\lambda\leq t<2^{\lambda+1}$.

\begin{conjecture}[Cusick, simplified]
For all $t\geq 0$, we have $c_t+c_{t'}>1$.
In other words, at least one out of $t$ or $t'$ satisfies Cusick's conjecture.
\end{conjecture}

Our main theorem is an approximation to this simplified conjecture.
\begin{theorem}\label{thm_main}
Assume that $\varepsilon>0$. There exists a constant $C=C(\varepsilon)$ such that
\[c_t+c_{t'}>1-\varepsilon,\]
if the binary expansion of $t$ contains at least $C$ blocks of consecutive $\mathtt 1$s.
\end{theorem}
We note that an admissible value of $C(\varepsilon)$ can be made completely explicit.
Note also that this theorem gives a lower bound $c_t>1/2-\varepsilon$ for 
many values of $t$:
in fact, the number of integers $0\leq t<T$ having less than $C$ blocks of consecutive $\mathtt 1$s in its binary expansion is bounded by $T^\eta$ for some $\eta<1$.
The important point is the fact that obtain a very efficient method of finding many $t$ such that $c_t>1/2-\varepsilon$: we only have to start with an integer having sufficiently many blocks of $\mathtt 1$s and possibly invert the digits between the first and the last $\mathtt 1$ (corresponding to $t\mapsto t'$) in order to arrive at such a $t$.

The remainder of this paper is dedicated to the proof of Theorem~\ref{thm_main}.
Throughout the proof, we use the common notations $\e(x)=\exp(2\pi i x)$ and $\lVert x\rVert=\min_{k\in\mathbb Z}\lvert x-k\rvert$.
\section{Proof of the main theorem}
For $t\geq 0$ and $k\in\mathbb Z$, we define the densities
\[\delta(k,t)=
\dens\{n\in\mathbb N:s(n+t)-s(n)=k\}.\]
Again, these densities exist~\cite{B1972}.
These values satisfy the recurrence~\cite{DKS2016}
\begin{equation*}
\begin{aligned}
\delta(k,1)&=\begin{cases}2^{k-2},&k\leq 1;\\0&\mbox{otherwise;}\end{cases}
\\
\delta(k,2t)&=\delta(k,t);\\
\delta(k,2t+1)&=\frac 12 \delta(k-1,t)+\frac 12\delta(k+1,t+1).
\end{aligned}
\end{equation*}
Moreover, we define a simplified array $\varphi$ by modifying the start vector:
\begin{equation*}
\begin{aligned}
\varphi(k,1)&=\begin{cases}1,&k=0;\\0&\mbox{otherwise;}\end{cases}
\\
\varphi(k,2t)&=\varphi(k,t);\\
\varphi(k,2t+1)&=\frac 12 \varphi(k-1,t)+\frac 12\varphi(k+1,t+1).
\end{aligned}
\end{equation*}
The reason for the introduction of this array, and in fact also the reason for the definition of $t'$, is the following symmetry property~\cite{DKS2016}: we have
\[\varphi(k,t)=\varphi(-k,t').\]
Moreover, by linearity we have
\[\delta(k,t)=\sum_{\ell+s=k}\varphi(\ell,t)\delta(s,1)
=\sum_{\ell\geq 0}\varphi(k+1-\ell)2^{-\ell-1}.\]

We obtain
\begin{equation}\label{eqn_fundamental}
\begin{aligned}
c_t+c_{t'}&=\sum_{\ell\geq -1}\bigl(\varphi(\ell,t)+\varphi(-\ell,t)\bigr)
\left(1-2^{-\ell-2}\right)\\
&=
\frac 32 \varphi(0,t)+\frac {11}8 \varphi(1,t)+\frac {11}8 \varphi(-1,t)
\\&+\sum_{\ell\geq 2}\left(1-2^{-\ell-2}\right)\bigl(\varphi(\ell,t)+\varphi(-\ell,t)\bigr).
\end{aligned}
\end{equation}

\begin{remark}
From the above identity we immediately obtain $c_t+c_{t'}\geq 15/16$ by using the identity $\sum_{k\in\mathbb Z}\varphi(k,t)=1$;
it is obvious to suspect that the maximum of $\varphi(k,t)$ is attained for $\lvert k\rvert\leq 1$. This would yield $c_t>1/2$ or $c_{t'}>1/2$ and thus settle the simplified form of Cusick's conjecture.
However, this assumption is wrong: for $t=149$ the maximum is attained at the position $k=2$. 
Still, there is hope: in fact, it is sufficient to prove that 
\begin{equation}\label{eqn_sufficient}
\varphi(-1,t)+\varphi(0,t)+\varphi(1,t)\geq\varphi(k,t)
\quad\mbox{ for all $k$ such that $\lvert k\rvert\geq 2$}.
\end{equation}
This can be seen as follows:
under this hypothesis we have
\begin{align*}
c_t+c_{t'}&\geq \sum_{\lvert \ell\rvert\geq 2}\varphi(\ell,t)-\sum_{\lvert \ell\rvert\geq 2}2^{-\lvert \ell\rvert-2}\max_{\lvert \ell\rvert\geq 2}\varphi(\ell,t)\\
&+\varphi(-1,t)+\varphi(0,t)+\varphi(1,t)+\frac 38\left(\varphi(-1,t)+\varphi(0,t)+\varphi(1,t)\right)
\\&\geq 1+\max_{\lvert\ell\rvert\geq 2}\varphi(\ell,t)\left(\frac 38-\sum_{\lvert\ell\rvert \geq 2}2^{-\lvert \ell\rvert-2}\right)>1.
\end{align*}
It looks like a simple thing to prove~\eqref{eqn_sufficient} by induction on the length of the binary expansion of $t$, using the recurrence for $\varphi$; however, so far we did not succeed.
\end{remark}
We are going to work with the following expression, where $\vartheta\in\mathbb R$.
\begin{equation}\label{eqn_correlation}
\omega_t(\vartheta)
=
\sum_{k\in\mathbb Z}\varphi(k,t)\e(k\vartheta).
\end{equation}
Obviously, this sum is absolutely convergent.

Our strategy is to give an upper bound for $\lvert\omega_t(\vartheta)\rvert$, where $\vartheta=j/m$.
This will give us some information on the behaviour of $k\mapsto \varphi(k,t)$ on residue classes $b+m\mathbb Z$.
For each $b$, we take the least weight appearing in~\eqref{eqn_fundamental} for $\ell\in b+m\mathbb Z$ and multiply it with the sum
$\sum_{\ell\in b+m\mathbb Z}\varphi(\ell,t)$;
afterwards, we sum the contributions of the different residue classes, using the argument on the smallness of $\omega_t(j/m)$.
 
\begin{lemma}
We have the following recurrence for the values $\omega_t(\vartheta)$.
\begin{equation*}
\begin{aligned}
\omega_1(\vartheta)&=1,\\
\omega_{2t}(\vartheta)&=\omega_t(\vartheta),\\
\omega_{2t+1}&=\frac {\e(\vartheta)}2\omega_t(\vartheta)+\frac {\e(-\vartheta)}2\omega_{t+1}(\vartheta).
\end{aligned}
\end{equation*}
\end{lemma}
The proof is straightforward, using the recurrence for $\varphi$.
This new recurrence can be written using $2\times 2$-matrices~\cite{MS2012}:
define 
\[A_0=\left(\begin{matrix}1&0\\\e(\vartheta)/2&\e(-\vartheta)/2\end{matrix}\right)
\quad\mbox{and}\quad
A_1=\left(\begin{matrix}\e(\vartheta)/2&\e(-\vartheta)/2\\0&1\end{matrix}\right).
\]

If $t=(\varepsilon_\nu\cdots \varepsilon_0)_2$ is the binary representation of $t\geq 1$, we have
\[
\omega_t(\vartheta)=
\left(\begin{matrix}1&0\end{matrix}\right)
A_{\varepsilon_0}\cdots A_{\varepsilon_{\nu-1}}
\left(\begin{matrix}1\\1\end{matrix}\right).
\]

\begin{lemma}\label{lem_omega_estimate}
Assume that the binary expansion of $t\geq 1$ contains at least $2M+1$ blocks of consecutive $\mathtt 1$s.
Then
\[\lvert \omega_t(\vartheta)\rvert\leq \left(1-\frac 12\lVert \vartheta\rVert^2\right)^M.\]
\end{lemma}
\begin{proof}
There are at least $M$ many positions $j\in\{0,\ldots,\nu-3\}$, having distance $\geq 3$ from each other, such that $(\varepsilon_j,\varepsilon_{j+1},\varepsilon_{j+2})=(\mathtt 1\mathtt 0\mathtt 0)$ or $(\varepsilon_j,\varepsilon_{j+1},\varepsilon_{j+2})=(\mathtt 1\mathtt 0\mathtt 1)$.
We show that for such a block we gain a factor of $1-\lVert \vartheta\rVert^2/2$. 
Using the row-sum norm $\lVert\cdot\rVert_\infty$ for matrices, which is derived from the maximum norm for vectors and which is sub-multiplicative,
is is sufficient to prove that 
\[\lVert A_1A_0A_0\rVert_\infty\leq 1-\frac 12\lVert \vartheta\rVert^2\quad\mbox{and}\quad \lVert A_1A_0A_1\rVert_\infty\leq 1-\frac 12\lVert \vartheta\rVert^2.\]
After a short calculation we obtain
\[A_1A_0A_0
=
\left(\begin{matrix}\frac{\e(\vartheta)}2+\frac 14+\frac{\e(-\vartheta)}8&\frac{\e(-3\vartheta)}8\\[2mm]
\frac{\e(\vartheta)}2+\frac 14&\frac{\e(-2\vartheta)}4
\end{matrix}\right)
\quad\mbox{and}
\]
\[A_1A_0A_1
=
\left(\begin{matrix}
\frac{\e(2\vartheta)}4+\frac{\e(\vartheta)}8&\frac 14+\frac{\e(-2\vartheta)}4+\frac{\e(-\vartheta)}8\\[2mm]
\frac{\e(2\vartheta)}4&\frac{\e(-\vartheta)}2+\frac 14
\end{matrix}
\right)
\]
and we see that the row-sum norm is strictly below $1$ as soon as $\vartheta\not\in\mathbb Z$.
More precisely, we use~\cite[Lemme~3]{D1972}, 
stating that
\[\left\lvert \frac 1q(1+z_1+\cdots+z_{q-1})\right\rvert
\leq 1-\frac 1{2q}\max_{1\leq j<q}(1-\Re z_j)
\]
for $\lvert z_1\rvert,\ldots,\lvert z_{q-1}\rvert\leq 1$.
For example, the left upper entry of $A_1A_0A_0$ can be bounded as follows:
\[
\left\lvert 
\frac{\e(\vartheta)}2+\frac 14+\frac{\e(-\vartheta)}8
\right\rvert
=
\frac 78
\left\lvert 
\frac 17\left(1+1+\e(-\vartheta)+4\cdot \e(\vartheta)\right)
\right\rvert
\leq \frac 78\left(1-\frac 1{14}\left(1-\Re \e(\vartheta)\right)\right).
\]

Analogously, the entry below, and also the right lower entry of $A_1A_0A_1$ may be bounded by
\[
\frac 34
\left\lvert
\frac 13
\left(
1+2\cdot \e(\vartheta)
\right)
\right\rvert
\leq \frac 34\left(1-\frac 16\left(1-\Re \e(\vartheta)\right)\right),
\]
while the right upper entry of this second matrix can be bounded by
\[\frac 58\left(1-\frac 1{10}\left(1-\Re \e(\vartheta)\right)\right).\]

It follows that $\lVert B\rVert_\infty\leq 1-\frac 1{16}\left(1-\Re \e(\vartheta)\right)$
for $B\in\{A_1A_0A_0,A_1A_0A_1\}$.
Moreover, we have the elementary inequality $\Re \e(\vartheta)=\cos(2\pi \vartheta)\leq 1-8\lVert \vartheta\rVert^2$,
so that
$\lVert B\rVert\leq 1-\frac 12\lVert \vartheta\rVert^2$.
This proves the lemma.
\end{proof}

We are interested in the quantity
\[
\psi(b,m,t)=\sum_{\ell\in b+m\mathbb Z}\varphi(\ell,t).\]

Moreover, we define
\[
\tilde a_\ell=\begin{cases}
3/2,&\mbox{if }\ell=0;\\
11/8,&\mbox{if }\lvert \ell\rvert=1;\\
1-2^{-\lvert\ell\rvert-2},&\mbox{if }\lvert \ell\rvert\geq 2.
\end{cases}
\]
Clearly $\tilde a_\ell\geq a_\ell\coloneqq 1-2^{-\lvert \ell\rvert-2}$ for all $\ell\in\mathbb Z$.
By~\eqref{eqn_fundamental}, we obtain
\begin{equation}\label{eqn_decomposition}
c_t+c_{t'}
\geq
\sum_{0\leq b<m}
\psi(b,m,t)
\min_{\ell\in b+m\mathbb Z} a_\ell.
\end{equation}

By monotonicity and symmetry of $a_\ell$, we obtain
\begin{equation}\label{eqn_a_ell_estimate}
\min_{\ell\in b+m\mathbb Z}a_\ell
=
\begin{cases}
1-2^{-b-2}&\mbox{if }0\leq b<m/2\\
1-2^{-(m-b)-2}&\mbox{if }m/2\leq b<m.
\end{cases}
\end{equation}
Moreover,
\begin{align*}
\psi(b,m,t)=
\sum_{\ell\in b+m\mathbb Z}
\varphi(\ell,t)
&=
\sum_{k\in \mathbb Z}
\varphi(k,t)
\frac 1m\sum_{0\leq j<m}
\e\left(j\frac{k-b}m\right)
\\
&=
\frac 1m\sum_{0\leq j<m}
\e\left(-\frac {jb}m\right)
\omega_t(j/m).
\end{align*}

By Lemma~\ref{lem_omega_estimate} it follows (using the abbreviation $x\pm y$ to stand for $x+O(y)$ with an implied constant $1$) that
\begin{equation}\label{eqn_psi_estimate}
\begin{aligned}
\psi(b,m,t)&=\frac 1m\pm\max_{1\leq j<m}\lvert \omega_t\left(j/m\right)\rvert
\\&=
\frac 1m\pm\left(1-1/(2m^2)\right)^M
=\frac 1m\pm e^{-M/(2m^2)},
\end{aligned}
\end{equation}
if $t$ has at least $2M+1$ blocks of consecutive $\mathtt 1$s in its binary expansion.

From~\eqref{eqn_decomposition} and~\eqref{eqn_psi_estimate} it follows that
\[c_t+c_{t'}
\geq 
\frac 1m\sum_{0\leq b<m}\min_{\ell\in b+m\mathbb Z}a_\ell
\pm me^{-M/(2m^2)}.
\]

It remains to consider mean values of the quantity in~\eqref{eqn_a_ell_estimate}. It is obvious that this mean value converges to $1$ for $m\rightarrow\infty$; quantitatively, we get for all $N\leq m$
\begin{align*}
\sum_{0\leq b<m}\min_{\ell\in b+m\mathbb Z}a_\ell
&\geq
\sum_{N\leq b<m-N}
\min_{\ell\in b+m\mathbb Z}a_\ell
\geq
(m-2N)\left(1-2^{-N-2}\right)\\
&\geq 
m\left(1-2^{-N-2}\right)-2N.
\end{align*}

We obtain 
\[c_t+c_{t'}\geq 1-2^{-N-2}-\frac{2N}m-me^{-M/(2m^2)}.\]

Let $\varepsilon\in(0,1)$ be given. We aim to define a bound $C$ as in the statement of the theorem.
Let $N=\lfloor - \log_2 \varepsilon\rfloor$+1, then clearly $2^{-N-2}<\varepsilon/3$.
Moreover, choose $m=\lfloor 6N/\varepsilon\rfloor +1$, then $2N/m<\varepsilon/3$.
Finally, let $M=\lfloor -2m^2\log(\varepsilon/(3m))\rfloor+1$, then  $me^{-M/(2m^2)}<\varepsilon/3$.
The choice $C=2M+1$ satisfies the claim of the theorem.
Asymptotically, an admissible choice for $C$ is given by 
$\alpha (\log \varepsilon)^3/\varepsilon^2$ for some constant $\alpha<0$ that can be given explicitly.

\begin{remark}
It would be desirable to improve our theorem in one or more of the following three aspects:
\begin{enumerate}
\item Prove statements on individual $c_t$ instead of the combined quantity $c_t+c_{t'}$.
\item Eliminate the quantity $\varepsilon$ appearing in our lower bound.
\item Prove statements for \emph{all} $t\geq 0$ instead of demanding the existence of many blocks of $\mathtt 1$s in the binary expansion of $t$.
\end{enumerate}
Of course, Cusick's original conjecture corresponds to (1)$\wedge$(2)$\wedge$(3), while the simplified form given above corresponds to (2)$\wedge$(3).

We expect that progress on (1) can be made by appealing to the study of moments of the probability distribution defined by $k\mapsto \delta(k,t)$ initiated by Emme and Prikhod'ko~\cite{EP2017} and pursued by Emme and Hubert~\cite{EH2018,EH2018b}.
This will be the subject of a future research paper.
\end{remark}

\bibliographystyle{siam}
\bibliography{M}
\end{document}